\documentclass[12pt]{amsart}

\newtheorem{theorem}{Theorem}[section]

\newtheorem{corollary}[theorem]{Corollary}

\newtheorem{lemma}[theorem]{Lemma}

\theoremstyle{definition}

\newtheorem{definition}[theorem]{Definition}

\newtheorem{example}[theorem]{Example}

\newcommand{\N}{\mathbb{N}}

\newcommand{\R}{\mathbb{R}}

\begin{document}
	
	\title[]{Boundedness of some convolution-type operators on metric measure spaces}

    
\author{J. M. Aldaz}
\address{Instituto de Ciencias Matem\'aticas (CSIC-UAM-UC3M-UCM) and Departamento de 
Mate\-m\'aticas,
Universidad  Aut\'onoma de Madrid, Cantoblanco 28049, Madrid, Spain.}
\email{jesus.munarriz@uam.es}
\email{jesus.munarriz@icmat.es}


\thanks{The author was partially supported by Grant PID2019-106870GB-I00 of the
MICINN of Spain, and also  by ICMAT Severo Ochoa project 
CEX2019-000904-S (MICINN)}







\begin{abstract} We explore boundedness properties in the context of metric measure spaces, of some natural operators of convolution type whose study is suggested by certain transformations used in computer vision.
\end{abstract}

	\maketitle
	
	
	\section {Introduction}

Using the class of indicator functions of balls in a metric measure space, it is possible to define a (restricted) notion of convolution as follows. Recall that for $x\in\R^d$ and a locally integrable $f:\R^d\to\R$,
$$
f*\mathbf{1}_{B(0,r)}(x) := \int_{\R^d} f(y) \mathbf{1}_{B(0,r)}(x - y) dy = \int_{B(x,r)} f(y)  dy.
$$
While general metric spaces lack a group structure, the second equality can nevertheless be used to define convolutions when the second terms belong the class of indicator functions of balls.

	\begin{definition} A Borel measure is   {\em $\tau$-additive} or {\em $\tau$-smooth}, if for every
		collection  $\{U_\alpha : \alpha \in \Lambda\}$
		of  open sets, 
		$$
		\mu (\cup_\alpha U_\alpha) = \sup_{\mathcal{F}} \mu(\cup_{i=1}^nU_{\alpha_i}),
		$$
		where the supremum is taken over all finite subcollections $\mathcal{F} = \{U_{\alpha_1}, \dots, U_{\alpha_n} \}$
		of  $\{U_\alpha : \alpha \in \Lambda\}$.
		We say that $(X, d, \mu)$ is a {\em metric measure space} if
		$\mu$ is a  $\tau$-additive  Borel measure on the metric space $(X, d)$, such that $\mu$ assigns finite measure to bounded Borel sets.
	\end{definition} 
	
For the purposes of the present paper, the reader not familiar with the notion of
$\tau$-additivity may assume without loss that
we are dealing with Borel measures on a separable metric space. A justification of this statement will be provided after Definition \ref{supp} (cf. also the last Section).

\begin{definition}\label{conv} Let  $(X, d, \mu)$ be a metric measure space and let $g$ be  a locally integrable function 
	on $X$. The
{\em metric convolutions}, or  {\em metric cross-correlations}, associated to the family of balls, are defined as follows: for each fixed  $r >0$ and each $x\in X$, set
	\begin{equation}\label{convop}
		g * \mathbf{1}_{B(x, r)} :=  \int _{B(x, r)}  g (y) \ d\mu (y).
	\end{equation}
\end{definition}
	
The motivation to study such objects comes from specific examples in the area of image processing:
 discretize the plane into pixels, and select a rectangle. Grey scale images can be defined by assigning to each
	pixel in the rectangle a value between 0 and 255 (from lesser to greater luminosity, so white corresponds to 255), and the value 0 to pixels outside the rectangle (the so called zero-padding). In computer vision (cf. for instance \cite{Sz}), an often employed step consists in fixing a certain length, say $2k + 1$, and replacing the value of each pixel with the sum of all the pixels in the square of sidelength $2k + 1$, centered at the given pixel  (and interpreting as white any value larger than 255). This operation is called ``convolution" or more precisely ``cross-correlation'' (since no reflection is involved; we will use both terms interchangeably). Note that ``convolution" is carried out only in a subset of the plane, the selected rectangle, so it does not coincide with the usual notion of convolution on $\R^2$. In this particular case cross-correlation is easily seen to define an $L^1$-bounded operator, since all pixels are given the same weight, so all squares have the same measure. However, it is natural to ask about boundedness of such operators in the general setting of metric measure spaces, as per Definition \ref{conv}. Regarding the cross-correlation $$
\int_{\R^d} f(y) \mathbf{1}_{B(0,r)}(x + y) dy = \int_{B(-x,r)} f(y)  dy
$$
on $\R^d$, there is no notion of reflection $x\mapsto -x$ in general metric spaces, so we do not distinguish between convolution and cross-correlation in Definiton \ref{conv}.

\vskip .2 cm

	There are convolutions with filters in image processing which do not quite fit into the above setting, for instance, the gaussian smoothing filter  given by the matrix
	$$
	G:= \frac{1}{16}\left(
\begin{array}{ccc}
  1 & 2 & 1 \\
  2 & 4 & 2 \\
  1 & 2 & 1 \\
\end{array}
\right).
$$
On each pixel $x$ in the selected rectangle we carry out the following operations. Take the $3 \times 3$ matrix $A$ centered at $x$, perform the Hadamard multiplication of matrices with $G$ and $A$, that is, the entry-wise product $GA$ of matrices. Then add up all the entries of $GA$ and replace $x$ with the value so obtained. Note that for the particular example where all entries of a filter $G$ are equal to 1, we are in the case generalized by Definition \ref{conv}.

\vskip .2 cm

 Next we consider this operation with arbitrary $G$, in the context of metric measure spaces.

\begin{definition}\label{convfilter} Let  $(X, d, \mu)$ be a metric measure space and let $g$ be  a locally integrable function 
	on $X$. The
{\em  metric convolutions with a filter}, or  {\em metric cross-correlations with a filter}, associated to the family of balls, are defined as follows: fix $r > 0$; for each  $x\in X$, let $\mu_{x,r}$ be a signed, $\tau$-additive measure on the ball $B(x,r)$. Then
	\begin{equation}\label{convopfilter}
		g * \mu_{x, r} :=  \int _{B(x, r)}  g (y) \ d\mu_{x, r} (y).
	\end{equation}
\end{definition}

The collection $\{\mu_{x,r}\}_{x\in X}$ is what we are calling a filter. 
\vskip .2 cm

	Boundedness for averages in metric measure spaces has been studied by the author in \cite{Al1} and \cite{Al2}. As we will see later, boundedness for metric cross-correlations $g * \mathbf{1}_{B(\cdot, r)} $ follows directly from the case for averages, which hold under a very weak Besicovitch type condition (cf. Definition \ref{ERBIP}). 	
Under an additional necessary condition on the relation between the measures in the filter and $\mu$,  we will show that 
boundedness properties of $g * \mu_{\cdot, r}$ can be derived from those
of $g * \mathbf{1}_{B(\cdot, r)} $. 

Regarding the motivating examples of rectangles in the plane, by the
Besicovitch covering theorem these hypotheses are always satisfied on $\R^d$ and
all its subsets. The interpretation of boundedness for these examples would be that the luminosity of the transformed images after applying a filter, is controlled by the luminosity of the original images.
Finally, we consider signed and not just positive measures in the preceding definition, because of the use for edge detection of filters such as the Prewitt filters, for example
$$
\left(
\begin{array}{ccc}
  - 1 & 0 & 1 \\
  -1 & 0 & 1 \\
 - 1 & 0 & 1 \\
\end{array}
\right).
$$

	\section {Additional definitions, and results} 
	
	We will utilize $B^{o}(x,r) := \{y\in X: d(x,y) < r\}$ to denote metrically open balls, 
	and 
	$B^{cl}(x,r) := \{y\in X: d(x,y) \le r\}$ to refer to metrically closed balls (``closed ball" will always be understood in the metric, not the
	topological sense). 
	If we do not want to specify whether balls are open or closed,
	we write $B(x,r)$. But when we utilize $B(x,r)$, we assume that all balls are of the same kind, i.e., all open or all closed.  

\vskip .2 cm	
		
	Recall that 
	the complement of
	the support $(\operatorname{supp}\mu)^c := \cup \{ B^{o}(x, r): x \in X, \mu B^{o}(x,r) = 0\}$
	of a Borel  measure
	$\mu$,  is an open set, and hence measurable. 
	
	\begin{definition}\label{supp} Let $(X, d)$ be a metric space and let
		$\mu$ be a locally finite Borel measure on $X$. 
		If $\mu (X \setminus \operatorname{supp}\mu) = 0$, 
		we say that $\mu$ has {\em full support}. 
	\end{definition}
	
	By $\tau$-additivity,  if $(X, d, \mu)$ is  a metric measure space, then 
	$\mu$ has full support,
	since $X \setminus \operatorname{supp}\mu $ is a union of open balls of measure zero.
	Actually, the other implication also holds, for the support of a $\tau$-additive locally bounded measure is always separable,  so  having full support is equivalent to
	$\tau$-additivity (cf. \cite[Proposition 7. 2. 10]{Bo} for more details). 	
Thus, if we are working with {\em just one measure}, we can suppose, by disregarding the set of measure zero $X \setminus \operatorname{supp}\mu $, that $X$ is separable and that all balls have positive measure. Furthermore, when we consider (possibly uncountable) collections of measures, as in Definition \ref{convfilter},  
the ``adapted'' condition from Definition \ref{adapted} ensures that the supports of all these measures are contained in $\operatorname{supp}\mu $. 
	
		\begin{definition}\label{aver} Let  $(X, d, \mu)$ be a metric measure space and let $g$ be  a locally integrable function 
	on $X$. The
	averaging operators $A_{r, \mu}$ acting on $g$ are defined as follows: for each  $ r > 0$ and each $x\in \operatorname{supp}\mu$, set
	\begin{equation}\label{avop}
		A_{r , \mu} g(x) := \frac{1}{\mu
			(B(x, r))} \int _{B(x, r)}  g \ d\mu.
	\end{equation}
\end{definition}
	
	Averaging operators in metric measure spaces are defined almost everywhere,   by
	$\tau$-additivity. 
	
	\vskip .2 cm
		
The next definitions come from \cite{Al2}. Recall that in general metric spaces centers and radii of balls are not unique,
that is, a ball (a set of points) may admit different descriptions or names.

\begin{definition}  \label{cover} 
	Let  $(X, d)$ be a   metric space, and let  $\mathcal{C}$ be a collection of balls. We say that   $\mathcal{C}$ is {\em uniformly bounded}  if there exists an $R > 0$ such that given any ball
	$B \in \mathcal{C}$, we can find a center $x$ and a radius $r$ with $B = B(x,r)$ and $r \le R$.
		
	Additionally,
	 we say   $\mathcal{C}$ is a {\em centered cover} of $A$ if  for every $x\in A$,  there exist a ball
	$B \in \mathcal{C}$ such that $x$ is a center of $B$.
\end{definition} 

As an illustration of Definition   \ref{cover}, take $X = (0,1)$ with the standard metric. The collection $\{(0,1)\}$ having only one ball  is a centered uniformly bounded cover of $(0,1)$. To see why, note that  we can select all radii $r = 1$, and then for all $x \in (0,1)$ we have  $B(x,1) = (0,1)$. It is also possible to select radii that are not uniformly bounded, for example, $\{(0,1)\} = \{B(x, 1/x) : x \in (0,1)\}$, but from the point of view of Definition   \ref{cover} this is irrelevant.

\begin{definition}  \label{ERBIP} A  collection $\mathcal{C}$ of balls in a metric space $(X, d)$ is a {\em Besicovitch family} if for each  $B \in \mathcal{C}$ it is possible to choose a center and a radius,
 such that whenever $B_1, B_2 \in \mathcal{C}$ are distinct balls, neither ball contains the selected center of the other.
 
Denote by
 	$\mathcal{E}(X,d)$ the collection of all Besicovitch families $\mathcal{C}$
 	of $(X, d)$ with the property that all balls in
 	$\mathcal{C}$ have equal radius.  That is, for each $B \in \mathcal{C}$ there is a choice of center
	and radius such that with that choice, $\mathcal{C}$ is a Besicovitch family and furthermore, all selected radii are equal. The 
 	{\em equal radius Besicovitch constant}  of $(X, d)$ is
 	\begin{equation}\label{BC}
 	E (X, d) := \sup \left\{	\sum_{B(x,  r) \in \mathcal{C}} \mathbf{1}_{B(x,  r)} (y): y \in X, \ \mathcal{C} \in \mathcal{E}(X,d) \right\}.
 	\end{equation}
 	We say that $(X, d)$ has the  {\em equal radius Besicovitch Intersection Property} with constant $E (X, d)$  if $E (X, d) < \infty$.
\end{definition} 
For example, if we consider the plane with the $\ell_\infty$ norm, so balls are just squares with sides parallel to the axes, it is known that $E (\R^2, \|\cdot\|_\infty) = 4$, cf.  \cite[Proposition 23]{Sw}.

\vskip .2 cm

 The preceding definition does not specify whether balls are open or closed, but there is no need to: it is shown in Proposition 4.6 of \cite{Al2} that
 a metric space $(X, d)$ has the equal radius  Besicovitch
 	intersection property with constant $E$  for collections of open balls, if and only if it
 	has   the equal radius  Besicovitch
 	intersection property for collections of closed  balls, with the same constant.

 Regarding the next result, it is part of \cite[Theorem 4.7]{Al2} plus a standard application of Jensen's inequality. 
 
 \begin{theorem}\label{constantE} Let $(X, d)$ be a  
 	metric  space.  If the space $(X, d)$ has the  equal radius Besicovitch intersection property with constant $E$, then
 for every $r > 0$ and every $\tau$-additive, locally finite  Borel measure $\mu$ on $X$, whenever $1\le p < \infty$ we have $\|A_{r, \mu}\|_{L^p  \to L^p} \le 
 	E^{1/p}$, and $\|A_{r, \mu}\|_{L^\infty  \to L^\infty} \le 1$.
 \end{theorem}
 
 \begin{proof} It follows from \cite[Theorem 4.7]{Al1} that $\|A_{r, \mu}\|_{L^1  \to L^1} \le 
 	E$   if $(X, d)$  has the  equal radius Besicovitch intersection property with constant $E$. Let $1 < p < \infty$ and let $f \in L^p$.
	Then $|f|^p \in L^1$, so by Jensen's inequality, 
	$$
	\|A_{r, \mu} f \|_{p}^p \le \|A_{r, \mu} |f|^p \|_{1} \le E \| |f|^p \|_{1} = E \| f \|_{p}^p.
	$$
	Of course, $
	\|A_{r, \mu} f \|_{\infty} \le  \| f \|_{\infty}
	$ is trivial.		
	\end{proof}
	
\begin{example}\label{bigballs} The metric convolution cannot be bounded independently of the measure of balls of fixed radius $r >0$.  
		
Let $(X, d)$ be $X := \{0, 1/2\}\times \N  \subset \R^2$ with the inherited distance. Let $\mu$ be the counting measure on $\{1/2\}\times \N$,  and on $X := \{0\}\times \N$, set $\mu \{(0,n)\}  = n$. 
	 Writing $f_n := \mathbf{1}_{\{(1/2,n)\}}$,
 for all $p\in [1, \infty]$ we have $\|f_n\|_{p} = 1$.
	 For convenience we shall use closed balls, with
	 $r = 1/2$, so when $n>0$, we have
 $f_n * \mathbf{1}_{B^{cl}((0,n), 1/2)}   = \int_{B^{cl}(\{(0,n\}, 1/2)} \mathbf{1}_{\{(1/2,n)\}} (y)\ d\mu(y) = 1$. 
 Thus, for $1 \le p < \infty$, 
 $$
 \|f_n * \mathbf{1}_{B^{cl}(\cdot, 1/2)}  \|_{p} 
 \ge
 n^{1/p}.
 $$
 For $p = \infty$, set $g_n := \mathbf{1}_{\{(0,n)\}}$. Then  $g_n * \mathbf{1}_{B^{cl}((0,n), 1/2)}   = \int_{B^{cl}(\{(0,n\}, 1/2)} \mathbf{1}_{\{(0,n)\}} (y)\ d\mu(y) = n$. 
\end{example}

If the measures of all balls of a fixed radius $r > 0$ are uniformly bounded above, then the metric convolution can be pointwise bounded by a suitable multiple of the corresponding averaging operator, so the bounds satisfied by the latter are inherited by the former. 

		\begin{lemma}\label{lemmapointwise} Let  $(X, d, \mu)$ be a metric measure space, let $R > 0$, and let $g$ be  a locally integrable function 
	on $X$.  Then for all $x\in \operatorname{supp}\mu$ we have  
	\begin{equation}\label{pointwiseconvop}
	|g| * \mathbf{1}_{B( x, R)} 
		\le  \sup_{y\in X} \mu B(y,R) A_{R , \mu} |g|(x).
	\end{equation}
	\end{lemma}

	\begin{proof} Suppose that for $R >0$ we have $\sup_{x\in X} \mu B(x,R) < \infty$; otherwise there is  nothing to show. For $x\in \operatorname{supp}\mu$,
	\begin{equation*}
		|g| * \mathbf{1}_{B(x, R)}
		=
 \frac{\sup_{y\in X} \mu B(y,R) }{\sup_{y\in X} \mu B(y,R) } \int _{B(x, R)}  |g| \ d\mu
 \end{equation*}
 \begin{equation*}
 \le 
  \frac{\sup_{y\in X} \mu B(y,R) }{\mu
			(B(x, R))} \int _{B(x, r)}  |g| \ d\mu
=  \sup_{y\in X} \mu B(y,R)  A_{R , \mu} |g|(x).
	\end{equation*}
		\end{proof}
	
\begin{corollary}\label{constantEconv} Let $(X, d, \mu)$ be a  
 	metric  measure space, let $g$ be locally integrable, and let $R > 0$.  If  $(X, d)$ has the  equal radius Besicovitch intersection property with constant $E$, then for $1\le p < \infty$ we have  
	\begin{equation}\label{opnormconvop}
		\|g * \mathbf{1}_{B(\cdot, R)}\|_{L^p (\mu)} 
		\le  
		\sup_{x\in X} \mu B(x,R) E^{1/p}  \|g\|_{L^p (\mu)},	
		\end{equation}
		and $\|g * \mathbf{1}_{B(\cdot, R)}\|_{L^\infty} \le \sup_{x\in X} \mu B(x,R)   \|g\|_{L^\infty (\mu)}$.
		 \end{corollary}

 \begin{proof} From Lemma \ref{lemmapointwise} and Jensen's inequality we get
 \begin{equation}
		 |g * \mathbf{1}_{B(x, R)}|^p
		\le  \left(\sup_{y\in X} \mu B(y,R) \right)^p A_{R , \mu} |g|^p (x).
	\end{equation}
	Then we integrate and apply Theorem \ref{constantE}. The case $p = \infty$ is trivial.
	\end{proof}

\begin{example}\label{tinyballs} In order to be able to obtain boundedness results for convolution with a filter, there has to be some relatioship between the filter $\{\mu_{x,r}\}_{x\in X}$  and the measure $\mu$. 

For convenience we shall use closed balls, with $r = 1$. 	
Let $(X, d)$ be $X := \{0, 1\}  \subset \R$ with the inherited distance, so $B^{cl}(x,1) = X$. Let $\mu := \delta_0$ be the point mass at 0, and for $x \in \{0,1\}$  let $\mu_{x,1} := \delta_0 + \delta_1$. Then $\mathbf{1}_{\{1\}}$ has $L^1(\mu)$-norm equal to zero, but $ \mathbf{1}_{\{1\}} * \mu_{0,1} = \int_{X} \mathbf{1}_{\{1\}} (y)\ d\mu_{0,r}(y) = 1$, so
$\|\mathbf{1}_{\{1\}} * \mu_{\cdot,1}\|_{L^1(\mu)} = 1$.
\end{example}

 Given a signed measure $\nu$, we denote by $|\nu|$ its total variation, that is, the sum of its positive and negative parts.

\begin{definition}\label{adapted} Let $(X, d, \mu)$ be a  
 	metric  measure space.   We say that a filter $\{\mu_{x,r}\}_{x\in X}$ is {\em adapted} to $\mu$ if there exists a constant $M < \infty$ such that for every $x, y \in X$ and every $0 < s \le r$, we have
$$
\frac{ |\mu_{x,r}| (B(y,s) \cap B(x,r))}{\mu (B(y,s) \cap B(x,r))}
\le M.
$$
 \end{definition}

\begin{lemma}\label{meastheoreticlemma} Let  $(X, d, \mu)$ be a metric measure space such that $(X, d)$ has the   equal radius Besicovitch intersection Property with constant $E  < \infty$, and let  $\{\mu_{x,r}\}_{x\in X}$ be
 a filter  adapted to $\mu$, with constant $M < \infty$. Fix $x\in X$ and $\varepsilon >0$. Then for every relatively closed set  $F\subset B(x,r)$ we have $|\mu_{x,r}| (F) < E M \mu (F) + \varepsilon$.
	\end{lemma}
	
 \begin{proof}  Fix $x\in X$ and $\varepsilon >0$, and let   $F\subset B(x,r)$ be a nonempty relatively closed  set. For notational convenience here we take $B(x,r)$ as the
 whole space, so open means relatively open in $B(x,r)$, and balls $B(y,s)$ mean $B(y,s) \cap B(x,r)$. 
 Let $O_n := \{u \in B(x,r): d(u, F) < 1/n\}$. Since $F = \cap_n O_n$, there is an $N\gg 1$ such that
  $\mu ( O_N ) < \mu (F) + \varepsilon /(EM)$.  
  
  Given the collection  
$\mathcal{C} := \{B(y,1/N): y\in  F\}$,
 there is a Besicovitch family $\mathcal{C}^\prime \subset \mathcal{C}$  
   which still covers $F$. The proof of this claim follows from a standard Zorn Lemma argument: consider the class  $\mathcal{P}$  of all subcollections of $\mathcal{C}$  that are Besicovitch families. Then $\mathcal{P}$  is nonempty, since every subcollection with just one ball is a Besicovitch family. Partially order $\mathcal{P}$   by inclusion, that is, by setting $\mathcal{C}_\alpha \le \mathcal{C}_\beta$ if and only if $\mathcal{C}_\alpha \subset \mathcal{C}_\beta$. Let $\Lambda$ be some index set and let $\{\mathcal{C}_\alpha\}_{\alpha \in \Lambda}$ be a chain in 
 $\mathcal{P}$. Suppose $B(u, 1/N), B(v, 1/N) \in \cup_{\alpha \in \Lambda} \mathcal{C}_\alpha$. Since $\{\mathcal{C}_\alpha\}_{\alpha \in \Lambda}$ is a chain, there is a $\mathcal{C}_\gamma \in \{\mathcal{C}_\alpha\}_{\alpha \in \Lambda}$ such that $B(u, 1/N), B(v, 1/N) \in \mathcal{C}_\gamma$, so neither $v \in B(u, 1/N)$ nor $u \in B(v, 1/N)$. Hence 
 $\cup_{\alpha \in \Lambda} \mathcal{C}_\alpha$ is a Besicovitch family, from which it follows that every chain in 
 $\mathcal{P}$ has an upper bound. By Zorn's lemma there is a maximal element $\mathcal{C}_m$ in $\mathcal{P}$.
 To see that $\mathcal{C}^\prime := \mathcal{C}_m$ covers $F$, note that if $y \in F \setminus \cup\mathcal{C}_m$, then $\mathcal{C}_m \cup \{B(y, 1/N)\}$ is also a Besicovitch family, contradicting the maximality of $\mathcal{C}_m$.

 Now
$$
\mathbf{1}_F
\le
 \mathbf{1}_{\cup_{B \in \mathcal{C}^\prime} B}
 \le	\sum_{B \in \mathcal{C}^\prime} \mathbf{1}_{B} 
\le E.
$$ 
If $
\sum_{B \in\mathcal{C}^\prime} \mathbf{1}_{B} (u)
> 0,
$ 
then $ u\in  \cup_{B \in \mathcal{C}^\prime} B$, so in fact 
$\sum_{B \in \mathcal{C}^\prime} \mathbf{1}_{B} 
\le
E \mathbf{1}_{\cup_{B \in \mathcal{C}^\prime} B}.$
Hence 
$$
|\mu_{x,r}| (F)
 \le	\int \left(\sum_{B \in\mathcal{C}^\prime}  \mathbf{1}_{B} \right)d |\mu_{x,r}| 
=
 \sum_{B \in \mathcal{C}^\prime} \int \mathbf{1}_{B} \ d |\mu_{x,r}|
 \le
 M \sum_{B \in \mathcal{C}^\prime} \int \mathbf{1}_{B} \ d \mu
 $$
 $$
 =
 M \int\left(\sum_{B \in \mathcal{C}^\prime}  \mathbf{1}_{B} \right)d \mu
 \le
 M \int E \mathbf{1}_{\cup_{B \in \mathcal{C}^\prime} B} \ d \mu
 \le
 EM \mu (O_N)
 \le
  EM \mu (F) + \varepsilon.
$$
	\end{proof}

 \begin{theorem}\label{constantEbis} Fix $r > 0$. Let $(X, d, \mu)$ be a  
 	metric  measure space such that $(X, d)$ has the  equal radius  Besicovitch intersection Property with constant $E  < \infty$, and let   $\{\mu_{x,r}\}_{x\in X}$ be a filter adapted to $\mu$, with constant $M < \infty$. Then for $1\le p < \infty$ and $g$ locally integrable we have  
	\begin{equation}\label{opnormconvop}
		\|g * \mu_{\cdot,r}\|_{L^p (\mu)} 
		\le  
		M \sup_{x\in X} \mu B(x,r) E^{1/p}  \|g\|_{L^p (\mu)},
		\end{equation}
		and $\|g * \mu_{\cdot, r}\|_{L^\infty} \le M \sup_{x\in X} \mu B(x,r)   \|g\|_{L^\infty (\mu)}$. \end{theorem}
 
 \begin{proof} Fix $x \in X$ and $r > 0$. We shall denote the restriction of $\mu$ to $B(x,r)$ also by $\mu$,
 and for the notation in this part of the argument, $B(x,r)$ will be regarded as the whole space.
 Note that  since we are dealing with the finite Borel measures $\mu$ and $|\mu_{x,r}|$ on the metric space  $B(x,r)$, for every measurable  $S\subset B(x,r)$ we have $\mu (S) = \sup_{F \subset S} \mu (F)$ and
 $|\mu_{x,r}| (S) =  \sup_{F \subset S} |\mu_{x,r}| (F)$, where the suprema are taken over the class of closed sets $F$ (contained in $S$).
By Lemma \ref{meastheoreticlemma},  if $F$ is closed and $\mu (F) = 0$, then  $|\mu_{x,r}| (F) = 0$, so $|\mu_{x,r}| \ll \mu$. Let $f_{x,r}$ be the Radon-Nikodym derivative of $|\mu_{x,r}|$ with respect to $\mu$. Note that $\mu$-a.e. $0\le f_{x,r} \le M$. The first inequality is trivial, since both measures are positive. 
For the second, consider a sequence of radii converging to zero, for instance, $\{r_n\}_{n=1}^\infty = \{1/n\}_{n=1}^\infty$. By Theorem \ref{constantE} together with \cite[Theorem 3.1]{Al1}, or directly by \cite[Corollary 4.10]{Al2},  for  every $f\in L^1(\mu)$,  
 	we have $\lim_{r_n\to 0}  A_{r_n} f  =  f$ in $L^1 (\mu)$. Thus, there is a subsequence $\{A_{r_{n_j}} f \}_{j =1}^\infty$ converging to $f$ almost everywhere.  Now suppose $\mu \{f_{x,r} > M\} > 0$.  Taking $f := f_{x,r}$ above, there is an $N \gg 1$ and a $y \in \{f_{x,r} > M\}$  such that 
 $$
\frac{ |\mu_{x,r}| (B(y,1/N))}{\mu (B(y,1/N))} = \frac{1}{\mu (B(y,1/N))} \int_{B(y,1/N)} f_{x,r} (u) \ d \mu (u) > M,
$$ 
since $f_{x,r} (y) > M$. This contradicts the hypothesis that the filter is adapted to $\mu$
with constant $M$.
But now we have the following pointwise control:
$$
|g * \mu_{x,r}| 
\le  
|g |*|\mu_{x,r}| 
 =	\int_{B (x, r)}  |g| \ d |\mu_{x,r}| 
 $$
 $$
 = \int_{B (x, r)}  |g|(y) f_{x,r} (y)\ d \mu (y) 
 \le
 M \int_{B (x, r)}  |g|  \ d \mu 
 =
 M |g|*\mathbf{1}_{B (x,r)}. 
 $$
Integrating and using Corollary \ref{constantEconv}, for $1\le p < \infty$ we have  
	\begin{equation}\label{opnormconvopfilter}
	\|g * \mu_{\cdot, r}\|_{L^p (\mu)} 
\le
M \|g * \mathbf{1}_{B(\cdot, r)}\|_{L^p (\mu)} 
		\le  
M \sup_{x\in X} \mu B(x,r) E^{1/p}  \|g\|_{L^p (\mu)},	
		\end{equation}
		and $ \|g * \mu_{\cdot, r}\|_{L^\infty (\mu)} 
		\le
M \|g * \mathbf{1}_{B(\cdot, r)}\|_{L^\infty} \le M \sup_{x\in X} \mu B(x,r)   \|g\|_{L^\infty (\mu)}$ for $p = \infty$.
	\end{proof}
  
  \section {Some remainders on the notion of $\tau$-additivity} 
  As requested by an anonymous referee (whose contribution is hereby acknowledged) we add some explanations on the role of $\tau$-additivity for general metric spaces with a Borel measure. While an uncountable union of 
  measurable sets may fail to be measurable,  when the measurable sets are open all
  uncountable unions are open, and hence Borel measurable. The $\tau$-additivity condition is a strengthening
  of countable additivity for open sets, which allows us to approximate arbitrary unions of open sets with finite
  unions and  arbitrarily small errors.
  
  If the metric space is separable, then uncountable unions
  of open sets can be reduced to countable unions, and hence $\tau$-additivity automatically holds for
  all Borel measures $\mu$. In particular, this is the case for Polish spaces.
  Separability is sometimes assumed when defining metric measure spaces (cf. for
  instance \cite[p.  62]{HKST}) but this has the consequence of a priori 
  excluding from the definition spaces that appear naturally in analysis, such as,
  for instance, $L^\infty ([0,1])$.
  While the hypothesis of $\tau$-additivity seems more natural (to this author) in the definition of general metric measure spaces, for the purposes of the present paper it can always be assumed that $X$ is separable, for the reasons indicated above.
  
  Also, if $X$ is not separable
  but $\mu$ is Radon (inner regular with respect to the compact sets) then $\mu$ is $\tau$-additive, since arbitrary open sets can be approximated from below by compact sets. 
  
  Given that $\tau$-additive measures simultaneously generalize Radon measures and Borel measures on separable metric spaces, it is natural to ask whether all Borel measures on metric spaces
  must be $\tau$-additive. It can be assumed as a new axiom that the answer is positive without risking additional inconsistencies, in which case neither separability nor $\tau$-additivity are needed in the definition of metric measure space. More precisely, it is known that if
  ZFC is consistent, we can add as an axiom the non-existence of measurable cardinals and
  obtain a new consistent theory, within which all Borel measures on a metric space are $\tau$-additive,
  cf. \cite[Proposition 7.2.10]{Bo}. So basically, the weakest condition one may want to consider in the definition of metric measure spaces is $\tau$-additivity.

\end{document}